\documentclass[12pt]{amsart}
\usepackage{xcolor}
\usepackage{amsfonts,amssymb,eucal}

\usepackage{amsthm} 
 \usepackage{amsmath} 
 \usepackage{latexsym}
 \usepackage{bbold,mathbbol,bbm}
 
 \numberwithin{equation}{section}

\newcommand{\mb}[1]{{\mbox{\boldmath{$#1$}}}}
\newcommand{\mc}[1]{{\mathcal{#1}}}
\newcommand{\got}[1]{{\mathfrak{#1}}}
\newcommand{\db}[1]{{\mathbb{#1}}}
\newcommand{\pa}{\partial}

  \newcommand{\F}{\ensuremath{\mathbb{F}}}                        
\newcommand{\oo}{\mathbb{I}}
\newcommand{\un}{{\oo_n}}
\newcommand{\tn}{\mathrm{n}}
\newcommand{\tm}{\mathrm{m}}

\newcommand{\zn}{{\mathbb{O}_n}}
\newcommand{\R}{\ensuremath{\mathbb{R}}}
\newcommand{\C}{\ensuremath{\mathbb{C}}}

\newtheorem{Remark}{Remark}

\newtheorem{Theorem}{Theorem}
\newtheorem{Proposition}{Proposition}
\newtheorem{lemma}{Lemma}

\theoremstyle{definition} 
\def\ii{\operatorname{i}}
\newcommand{\mr}[1]{{\mathrm{#1}}}

\newcommand{\dd}{\operatorname{d}}

\newcommand{\nn}{\nonumber}

\newcommand{\tr}{\operatorname{tr}}

\allowdisplaybreaks 
\newcommand{\Sp}{\ensuremath{{\mbox{\rm{Sp}}(n,\R)}}}
\newcommand{\cM}{\mathcal{M}}
\newcommand{\cH}{\mathcal{H}}

\newcommand{\twopartdef}[4]

\begin{document}
  
  \title[Linear Hamiltonians]{Linear Hamiltonians in  generators
  of the real Jacobi  group on the extended Siegel-Jacobi space
  and equations of motion attached}

\author{Elena Mirela Babalic,  Stefan  Berceanu}
\address[Elena Mirela Babalic, Stefan  Berceanu]{``Horia Hulubei'' National
 Institute for Physics and Nuclear Engineering\\
         Department of Theoretical Physics\\
         PO BOX MG-6, Bucharest-Magurele, Romania}
       \email{mbabalic@theory.nipne.ro,Berceanu@theory.nipne.ro}

       \enlargethispage{1cm}
        \begin{abstract}
  
 Using the energy function on the extended Siegel-Jacobi upper half
 space of order  $n$,  $\tilde{\mc{X}}^J_n$, with $n\in \mathbb{N}$,  the
 equations of motion in the variables $(x,y,q,p,\kappa)$ attached to  linear  Hamiltonians in the
 generators of the real Jacobi group $G^J_n(\R)$ are presented, where
 $x,y$ are symmetric matrices in  $\cM(n,\R)$ and $p,q$ are real $n$-vectors. The case $n=1$ is presented separately.
    
\end{abstract}
\maketitle
\tableofcontents

\vspace{-1em}

\textbf{M.S.C. 2020:}  22E30, 20G05,  11F50, 12E10, 81R30 

\

\section{Introduction}\label{sec1}

\

 We consider  $G^J_n:=H_n\rtimes\text{Sp}(n,\R)_{\C}$, 
 where $H_n$ denotes the $(2n+1)$-dimensional 
 Heisenberg group and ${\rm  Sp}(n,\R)_{\C}:= 
 {\rm Sp}(n,\C)\cap {\rm U}(n,n)$.

 The real Jacobi group of degree $n$ is defined as
 $G^J_n(\R):=H_n(\db{R})\rtimes \text{Sp}(n,\R).$

Let us  consider coherent states  based on
Siegel-Jacobi ball  $\mc{D}^J_n$ \cite{sbj}, which 
consists of the points of $\C^n\times\mc{D}_n$,
 where  $\mc{D}_n$ denotes the Siegel ball
\cite[page 4]{sbj}.

 Let us denote by $FC$  the change of variables 
 $FC: ~\C^n\times\mc{D}_n$
$(\eta,W)\rightarrow$ $ (z,W)\in \mc{D}^J_n, $ $z= \eta-W\bar{\eta}\,$, as in \cite{FC, nou}.

The Jacobi group $G^J_n$ is studied in Mathematics, 
 Mathematical Physics and Theoretical Physics 
 \cite{SB19}--\cite{SB17}.

 We denote be $\mc{X}_n$ the Siegel upper half-space \cite[p. 577]{helg},
  \cite{yang,Y08}, with dimension  $\dim{\mc{X}_n}=n(n+1)$.
The notion of extedened Sigel-Jacobi plane  $\tilde{\mc{X}}^J_1$ and, more generally, extedened Sigel-Jacobi space
$\tilde{\mc{X}}^J_n$,  was introduced in \cite{SB19} and respectively \cite{alb}.
The dimensions of the other manifolds enumerated above are: $\dim{\Sp}=2n^2+n$,
$\dim{{H}_n(\R)}=2n+1$,
  $\dim{G^J_n(\R)}=(2n+1)(n+1)$,
  $\dim{\mr{U}(n)}=n$,
  $\dim{\mc{X}^J_n}=n(n+3),$
    $\dim{\tilde{\mc{X}}^J_n}=n(n+3)+1$ (see \cite{JGP2}).

 The paper is organized as follows. After the introductive Section 1, in Section 2 we recall the energy function on $\tilde{\mc{X}}^J_1$  for linear Hamiltonians in generators of $G^J_1(\R)$ and  we give Lemma 1 extracted from  \cite{SB22} where we show how $(\tilde{\mc{X}}^J_1, \theta\!=\!\eqref{TT11}, \omega\!=\!\eqref{TT12})$ can be organized as an ACOS manifold $(\tilde{\mc{X}}^J_1, \theta\!=\!\eqref{t11}, \omega\!=\!\eqref{t12})$ and also as a GTACOS manifold. 
In Section 3 we write the energy function on $\tilde{\mc{X}}^J_n$ for linear Hamiltonians in generators of $G^J_n(\R)$ and we give Lemma 2 about associating to the manifold $(\tilde{\mc{X}}^J_n, \theta=\eqref{tn1}, \omega=\eqref{tn2})$ an ACOS/GTACOS structure. We note in Theorem 1 that Lemma 2 is in perfect correspondence with Lemma 1, the former giving actually a generalization of the case $n\!=\!1$ to the case $n\!>\!1$. In Section 4 we include Theorem 2 which gives the general formulas for the equations of motion on $\tilde{\mc{X}}^J_1$,  organized as an ACOS manifold, and Proposition 2 in which we concretely compute the equations of motion on $\tilde{\mc{X}}_1^J$ using the total energy function \eqref{H1-tilde} associated to the linear Hamiltonian \eqref{guru}. In Section 5 we give Theorem 3, in which we write the equations of motion for $\tilde{\mc{X}}^J_n$, with $n>1$, organized as an ACOS manifold, and Proposition 3, in which, using the correspondences given in Lemma 2, we write the general fomulas for the equations of motion in the case with $n>1$. Using the energy function \eqref{enf} associated to the linear Hamiltonian \eqref{HACA}, we compute the equations of motion and give their concrete results in Theorem 4 for the case $n>1$. The paper ends with a section of Conclusions and a brief Appendix where one can find, in particular, the definitions of ACOS and GTACOS manifolds.

\
        
\textbf{Notation.} We denote by $\mathbb{R}$, $\mathbb{C}$, $\mathbb{Z}$
and $\mathbb{N}$ the field of real numbers, the field of complex numbers,
the ring of integers, and the set of non-negative integers, respectively. 
 $\cM(n,\mathbb{F})=\cM_{mn}(\mathbb{F})\approxeq\mathbb{F}^{mn}$ denotes the set of all $m\times n $ matrices with entries in the field $\mathbb{F}$ and $\cM(n,1,\mathbb{F})=\cM_{n 1}(\mathbb{F})$ is
identified with $\mathbb{F}^n$.
For any $A\in \cM(n,\mathbb{F})$, $A^{t}$ denotes the transpose matrix of
$A$, while $A^s:=(A+A^t)/2$ and $A^a:=(A-A^t)/2$ denote the symmetrization and respectively the antisymmetrization of matrix $A$. For $A\in \cM_{n}(\mathbb{C})$, $\bar{A}$ denotes the conjugate matrix
of $A$ and $A^{*}:=\bar{A}^{t}$. For $A\in \cM_n(\mathbb{C})$, the
inequality $A>0$ means that $A$ is positive definite. The identity matrix of
degree $n$ is denoted by $\un$ and  $\zn$ denotes the $\cM_n(\F)$-matrix with all entries $0$. We denote by
$\text{diag}(\alpha_1,\dots,\alpha_n$)  the matrix which has the
elements $\alpha_1,\dots,\alpha_n$ on the diagonal and all the other
elements 0.
If $A$ is a linear operator, we denote by
$A^{\dagger}$ its adjoint. If $A, B$ are  operators, we use the classical notations for the comutator $[A,B]=AB-BA$, and for the anti-commutator $\{A,B\}=AB+BA$.
Notation $(q,p,\kappa)$ referes to the variables in the Heisenberg group $H_n(\R)$, for $n\geq 1$ (see \cite{JGP2}, \cite{SB15}, \cite[page 12]{ez}), where $q,p\in \cM(n,1,\R)$ and $ \kappa \in \R$. One should pay attentian not to confuse constant $k$ with coordinate $\kappa$, nor dimensions $m,n$ with matrices $\tm,\tn$. Due to space reasons, we may use notation $\overline{1,n}$ instead of $1,\ldots,n$.

\pagebreak

\section{Energy function and ACOS structure for $\tilde{\mc{X}}^J_1$}
  
  \
  
 Following \cite[(18)]{GMP3}, we consider a linear hermitian Hamiltonian in the generators of the Jacobi
group $G^J_1$ (see also \cite{SB22}):
\begin{subequations}
\label{guru}
\begin{eqnarray}
&& \mb{H} = \epsilon_a\mb{a} +\bar{\epsilon}_a\mb{a}^{\dagger}
 +\epsilon_0 {\mb{K}}_0 +\epsilon_+{\mb{K}}_++\epsilon_-{\mb{K}}_- ~ ,~~\\
\label{oareE}
&& \epsilon_a=a+\ii b, \quad
\bar{\epsilon}_+=\epsilon_-=  \tm+\ii \tn, 
\quad \epsilon_0=\bar{\epsilon}_0=2c,\quad a,b,c,\tm,\tn\in \R.
\end{eqnarray}
\end{subequations}

Following  \cite[(4.42)]{SB22},  we put into  evidence that the energy function attached to the
hermitian Hamiltonian (\ref{guru})  is real and write it as:
\begin{subequations}
\label{realHH}
\begin{eqnarray}
\mc{H}&=&\mc{H}_{\eta}+\mc{H}_{w}~,\\
~\mc{H}_{\eta}   & = &{\nu}\,[  \bar{\epsilon}_a\eta+\epsilon_a\bar{\eta} + 
\frac{1}{2}(\epsilon_+\eta^2+\epsilon_-\bar{\eta}^2+\epsilon_0\eta\bar{\eta})],\label{realHH1} \\
~ \mc{H}_{w}  & = &   k\epsilon_0+ \frac{2k}{1-w\bar{w}}(\epsilon_+w+\epsilon_-\bar{w}+\epsilon_0w\bar{w}), \label{realHH2}
\end{eqnarray}
\end{subequations}
where, according to a partial Cayley transform (see \cite{SB22}) we have: 
\begin{equation}
\label{wv}
w=\frac{v-\ii}{v+\ii},\quad v\in\C.
\end{equation}

 The energy function $\mc{H}$ on $\mc{X}^J_1$ expressed in variables $(\eta, v)$
 splits into  the sum of two independent functions:
\begin{equation}\label{H1-sum}
  \mc{H}(\eta,v)=\mc{H}(\eta)+\mc{H}(v),\quad v=x+\ii y, ~y>0,~\eta=q+\ii p, 
\end{equation}
where:
\begin{subequations}\label{sup}
\begin{align}
\mc{H}(\eta)=\mc{H}(q,p) &=\nu[(\tm+c)q^2+(c-\tm)p^2+2\tn qp+2(aq+bp)],\label{sup1}\\
\mc{H}(v)=\mc{H}(x,y) & =\frac{k}{y}[(\tm+c)(x^2+y^2)-2\tn x+c-\tm]. \label{sup2}
\end{align}
\end{subequations}

Extending $\mc{X}^J_1$ to $\tilde{\mc{X}}^J_1$ translates into adding an extra dimension to $\mc{X}^J_1$, in particular adding an independent coordinate $\kappa \in \R$.\footnote{One should not confuse the coordinate $\kappa$ with the constant $k$.} The total energy function on $\tilde{\mc{X}}^J_1$  writes:
\begin{equation}
\label{H1-tilde}
  \mc{H}_{\tilde{\mc{X}}^J_1}=\mc{H}(q,p)+\mc{H}(x,y)+\mc{H}(\kappa).
  \end{equation}

\begin{Remark} We show how we get from $H_{\eta}$ on $\mc{D}_1$ in \eqref{realHH1}   to $\mc{H}(q,p)$ on
  $\mc{X}^J_1$ in \eqref{sup1}.
\begin{proof}
  \begin{subequations}
 \begin{align*}
 \frac{\mc{H}_{\eta}(q, p)}{\nu} &= \bar{\epsilon}_a\eta+ \epsilon_a \bar{\eta}+\frac{1}{2}(\epsilon_+\eta^2 + \epsilon_-\bar\eta^2 + \epsilon_0\eta\bar{\eta})\\
 &= (a - \ii b)(q + \ii p) + (a + \ii b)(q - \ii p) \\
 & ~~~+\frac{1}{2}
\big[(m - \ii \tn)(q + \ii p)^2 + (\tm + \ii \tn)(q - \ii p)^2 + \epsilon_0(q^2 + p^2)\big]\\
& = 2(aq + bp) + \tm(q^2- p^2) + 2\tn qp +c(q^2+p^2)\\
                                 &= (\tm + c)q^2 + (c - \tm)p^2 + 2\tn qp + 2(aq + bp).
  \end{align*}
    \end{subequations}
\end{proof}
\end{Remark}

\begin{Remark} We show how we get from $\mc{H}_w $  defined in \eqref{realHH2}   to $\mc{H}_v$  in \eqref{sup2}.
  \begin{proof}
  \begin{subequations}
    \begin{align*}
    &  1 - w\bar{w} = 1 -\frac{v-\ii }{v+\ii}\frac{\bar{v}+\ii}{\bar{v}-\ii}=\frac{4y}{x^2+(y+1)^2}~,\\
& (1 - w \bar{w})^{-1} =\frac{x^2 + (y + 1)^2}{4y}~.
    \end{align*}
  \end{subequations}
 Let
 \begin{subequations}
    \begin{align*}   
      \mc{H}_v&= X_1+ X_2,
    \end{align*}
    \end{subequations}
    where
    \begin{subequations}
      \begin{align*}   
      X_2 &= k\epsilon_0=2kc,
\end{align*}
    \end{subequations}
    and
   \begin{subequations}
\begin{align*}   
X_1& =\frac{k}{2y}[x^2+(y +1)^2]\Big[\epsilon_+\frac{v - \ii}{v
+\ii}+ \epsilon_-\frac{\bar{v} + \ii}{\bar{v}-\ii}
 + \epsilon_0\frac{v -\ii}{v +\ii}\frac{\bar{v}
     +\ii}{\bar{v} -\ii}\Big]\\
   &=\frac{k}{2y}[x^2+(y +1)^2]\Big[\epsilon_+\frac{v -\ii}{v
   +\ii}\frac{\bar{v} -\ii}{\bar{v}-\ii}+c.c. +
     2c\frac{v -\ii}{v +\ii}\frac{\bar{v}
     +\ii}{\bar{v} -\ii}\Big]\\
  & =\frac{k}{2y}\big[(\tm-\ii \tn)(x^2+ y^2 -1-2\ii x)+c.c.+2c(x^2+y^2 +1-2y)\big].
\end{align*}
\end{subequations}
With \eqref{oareE}, $X_1$ becomes:
\begin{subequations}
      \begin{align*}  
X_1& =\frac{k}{y}(\tm+c)(x^2+y^2)-2(\tn x+cy)-\tm+c],
\end{align*}\end{subequations}  
 and finally we find:

 \begin{eqnarray}
 {\mc{H}_v}&=&\frac{k}{y}[(\tm + c)(x^2+y^2)-2(\tn x+cy)+c-\tm]+2ck.
\end{eqnarray}
  \end{proof}
\end{Remark}

We extract below from \cite[Lemma 1]{SB22}:
\begin{lemma}\label{L1}
 We endow the extended 5-dimensional Siegel-Jacobi upper half-plane   $\tilde{\mc{X}}^J_1$,
parametrized by $(x,y,q,p,\kappa)$, with an almost cosymplectic (ACOS) structure $(\tilde{\mc{X}}^J_1,\theta,\omega)$,
where the 1-form $\theta$ and the symplectic 2-form $\omega$ are given by:
\begin{subequations}
\label{TT1}
  \begin{align}
  \theta & =\sqrt{\delta}(\dd \kappa  -p \dd q+q\dd p)~,\quad \delta>0, \label{TT11}\\
    \omega & =\frac{k}{y^2} \dd x \wedge \dd y +2\nu \dd q\wedge \dd p
    ~,\quad y>0, \label{TT12}
   \end{align}
 \end{subequations}
 and satisfy condition \eqref{thtOM}, since:
 \begin{equation}
 \label{thOM}
\theta \wedge\omega^2=4\frac{k\nu\sqrt{\delta}}{y^2} \dd x\wedge\dd y\wedge\dd q\wedge\dd p \wedge \dd\kappa.
 \end{equation}
In the terminology of Libermann \cite{paul} recalled in Appendix in the definition of the ACOS structure $(M_{2N+1},\theta,\omega)$, one can write $\theta$ and $\omega$ for the case $(\tilde{\mc{X}}^J_1,\theta,\omega)$ also in the form:
\begin{subequations}
\begin{eqnarray}
\label{t11}
 && \theta=\sum_{I=1}^{\mr{N=2}}(a_I\dd Q^I+b_I \dd P^I)+c \dd \kappa, \quad a_I,b_I,c\in \R,  \quad c\neq 0,\\
 \label{t12}
  &&\omega  = \sum_{I=1}^{N=2}\dd Q_I\wedge \dd P^I.
  \end{eqnarray}
\end{subequations}
 Comparing \eqref{TT11} with \eqref{t11} and  \eqref{TT12} with \eqref{t12}, we can choose correspondences:
 \begin{subequations}
 \label{corr1}
\begin{eqnarray}
& c=\sqrt{\delta}~,&\quad a_1=b_1=0~,\quad \\
& a_2=-\frac{\sqrt{\delta}}{2\nu} p~,&\quad b_2=\sqrt{\delta}q ~, \\
& Q^1=kx~,&\quad P_1=-\frac{1}{y}~,\quad\\
&  Q^2=2\nu q~,&\quad P_2=p~.
 \end{eqnarray} 
 \end{subequations}
 
  The manifolod $(\tilde{\mc{X}}^J_1,\theta\!=\!\eqref{TT11},\omega\!=\!\eqref{TT12})$ can also be considered a GTACOS (generalized transitive almost cosymplectic) manifold
 since the following condition holds:
  \begin{equation}
  \label{cond0}
 \dd \omega=0 ~.
    \end{equation}
\end{lemma}

  \

\section{Energy function and ACOS structure for 
 $\tilde{\mc{X}}^J_n$, when $n>1$ }

\

\noindent We consider a linear hermitian Hamiltonian in the generators of the group $G^J_n$  (\cite[\S 8.2]{nou}):
\begin{equation}\label{HACA}
\mb{H}= \epsilon_i\mb{a}_i+\overline{\epsilon}_i\mb{a}_i^{\dagger} +  
\epsilon^0_{ij}\mb{K}^0_{ij}+
\epsilon^-_{ij}\mb{K}^-_{ij}+\epsilon^+_{ij}\mb{K}^+_{ij},
\end{equation}
where the hermiticity condition imposes on the matrices $\epsilon^{0,\pm}$ of  coefficients $(\epsilon^{0,\pm})_{i,j=1,\dots,n} $  the following restrictions:
\begin{equation}\label{CONDI}
\epsilon_0^{\dagger}=\epsilon_0, \quad \epsilon_-^t=\epsilon_-,\quad
\epsilon_+^t=\epsilon_+,\quad \epsilon_+^{\dagger}=\epsilon_-
\end{equation}
 and where we have:\footnote{Do not confuse dimension $n$ with matrix $\tn$. } 
\begin{subequations}\label{epsmn}
  \begin{align}
   &  \epsilon_0=2(A-\ii B),\quad A^t=A, ~ B^t=-B, \quad A,B\in\cM(n,\R), \label{epsmn0}\\
   & \epsilon_-=\tm+\ii \tn,\quad \tm^t= \tm, ~\tn^t =\tn,
  \quad \tm,\tn\in\cM(n,\R),     \label{epsmn1}\\
   & \epsilon_i^t=\epsilon_i=\alpha_i+\ii
                \beta_i,\quad i=1,\dots,n, \quad\alpha,\beta\in\cM(n,1,\R). \label{epsmn2}
    \end{align}
\end{subequations} 

\begin{Remark}
When  $n =1$ in $\cM(n,\R)$, we have $ A=c,~ B=0$ in \eqref{epsmn0} and $\alpha=a,~\beta=b$ in \eqref{epsmn2} to have correspondence with \eqref{oareE}.
\end{Remark}
 
 We calculate {\it{the energy function}}
  attached to the Hamiltonian (\ref{HACA}) and write it as a sum of two separated terms in the  independent  variables
$\eta\in\C^n$ and $W\in\mc{D}_n$:
\begin{subequations}\label{enf}
\begin{eqnarray}
\mc{H} & = &\mc{H}_{\eta}+ \mc{H}_{W},\\
\!\mc{H}_{\eta} & =&\nu\big[\epsilon^t\eta +
\bar{\epsilon}^t\bar{\eta}+\frac{1}{2}(\eta ^t\epsilon_- \eta
+\bar{\eta}^t\epsilon_+\bar{\eta} + 
                   \bar{\eta}^t\epsilon_0 \eta )\big],\label{Hn-eta}\\
  \mc{H}_{W} & = &\frac{k}{2}\tr\!\big\{(\epsilon_0)^s+
                     [W\epsilon_{-}+\epsilon_{+}\bar{W}+(\epsilon_0W)^s\bar{W}](\un-W\bar{W})^{-1}\big\}.\label{Hn-W}
    \end{eqnarray}                     
  \end{subequations}
  
   Passing from $\mc{X}^J_n$  to $\tilde{\mc{X}}^J_n$ translates into adding an extra dimension to the space, in particular adding a new independent coordinate $\kappa \in C^\R$.
 
\begin{Remark}   We write the energy function on  $\tilde{\mc{X}}^J_n$ associated to the linear Hamiltonian \eqref{HACA} as a sum of three independent energy functions:
\begin{subequations}\label{EF}
\begin{eqnarray}
  &&\mc{H}_{\tilde{\mc{X}}^J_n}(\eta,V,\kappa)=\mc{H}(\alpha,\beta,q,p)+\mc{H}(V)+\mc{H}(\kappa),~\\
  &&  \quad\quad \quad \eta=q+\ii p, ~V=x+\ii y,~ y>0~,
  \end{eqnarray}
\end{subequations}
where $\mc{H}(\alpha,\beta,q,p)$ is given by \eqref{bumber1}, while $\mc{H}(V)$ by \eqref{HV3}.
\end{Remark}

Given the relations:
    \begin{equation}\label{NR1}
      \epsilon_0^s=\frac{1}{2}(\epsilon_0+\epsilon_0^t)=\frac{2}{2}[A-\ii
      B+(A-\ii B)^t]=2A-\ii (B+B^t)=2A
    \end{equation}
    and
    \begin{equation}\label{NR2}
      (\epsilon_0W)^s\!=\!W(A\!-\!\ii B)\!+\!(A\!-\!\ii B )^tW\!=\!WA\!+\!AW\!-\ii (WB\!-\!BW)\!=\!\{W,A\}\!-\ii[W,B],
    \end{equation}
    equation  \eqref{Hn-W} becomes:
      \begin{eqnarray}
      \label{enfF}
   \quad\quad   \mc{H}_{W} \!\!\!& = & \!\!\! k\tr A +\frac{k}{2}\tr\big\{(W\epsilon_{-}+\epsilon_{+}\bar{W})(\un -W\bar{W})^{-1}\\
    &&+ [(WA+AW)\bar{W}+\ii (WB-BW)]\bar{W}(\un -W\bar{W})^{-1}\big\}\nn\\
      &=& k\tr A +\frac{k}{2}\tr\big\{[(\tm+\ii \tn)(\un
             -W\bar{W})^{-1}W+(\tm-\ii \tn)\bar{W}(\un-W\bar{W})^{-1}]\nonumber\\ 
       &&+ A[\bar{W}(\un-W\bar{W})^{-1}W+W\bar{W}(\un -W\bar{W})^{-1}]\nonumber\\ 
      &&- \ii B[\bar{W}(\un -W\bar{W})^{-1}W-W\bar{W}(\un -W\bar{W})^{-1}]\nonumber\big\}.
        \end{eqnarray}

  \begin{Remark}
    When $n=1$ equation \eqref{Hn-W} becomes:
   \begin{equation} 
   \mc{H}_{W}|_{n=1}=kc +\frac{k}{2}\frac{w(\tm+\ii \tn)+\bar{w}(\tm-\ii
       \tn)+2c w\bar{w}}{1-w\bar{w}},\label{MARX}
  \end{equation}
which reduces to relation \eqref{realHH2} of $\mc{H}_w$ (up to $k\to\frac{k}{4}$ due to the different conventions considered in the general expressions for the Hamiltonians \eqref{realHH} and \eqref{enf} respectively, taken as in the papers of reference \cite[(4.29)]{FC} for $n=1$  and \cite[(8.5)]{nou} for $n>1$). 
  \end{Remark}
 
\begin{Proposition}
\label{P1}
  If we make the change of coordinates $W\rightarrow V$,~~ $W=(V-\un)(V+\un)^{-1}$, with $V=x+\ii y$, where $x,y\in\cM(n,\R)$ are symmetric matrices, with $y$ being positive definite, i.e. $y>0$, and using notations \eqref{epsmn}, then the energy function in \eqref{enfF} gets the following expression:
  \begin{eqnarray*}
    &&\mc{H}_V = k\tr A
    +\!\!\frac{k}{2}\tr\Big\{
             \frac{m}{2\ii}[(\bar{V}\!-\!\ii\un)(\bar{V}\!-\!V)^{-1}(V\!+\!\ii\un)(V\!-\!\ii\un)(V\!+\!\ii\un)^{-1} \nn\\
     &&\quad\quad\quad\quad\quad\quad\quad\quad\quad+(\bar{V}\!+\!\ii\un)(\bar{V}\!-\!\ii\un)^{-1}(\bar{V}\!-\!\ii\un)(\bar{V}\!-\!V)^{-1}(V\!+\!\ii\un)]\Big\}\nn\\
    &&\quad+\frac{k}{4}\tr\big\{n[(\bar{V}-\ii\un)(\bar{V}\!-\!V)^{-1}(V\!+\!\ii\un)(V\!-\!\ii\un)(V\!+\!\ii\un)^{-1}\nn\\
    &&\quad\quad\quad\quad-(\bar{V}\!+\!\ii\un)(\bar{V}\!-\!\ii\un)^{-1}
      (\bar{V}\!-\!\ii\un)(\bar{V}\!-\!V)^{-1}(V\!+\!\ii\un)] \big\}\nn\\
    &&\quad+\frac{k}{4\ii}\tr\big\{A(\bar{V}\!\!+\!\ii\un)(\bar{V}\!-\ii\un)^{-1}
      (\bar{V}\!-\!\ii\un)(\bar{V}\!-\!V)^{-1}(V\!+\!\ii\un)(V\!-\!\ii\un)(V\!+\!\ii\un)^{-1}\big\}\nn\\
    &&\quad+\frac{k}{4\ii}\tr\big\{A(V\!-\!\ii\un) (V\!+\!\ii\un)^{-1}(\bar{V}\!+\!\ii\un)(\bar{V}\!-\!\ii\un)^{-1}(\bar{V}\!-\!\ii\un)(\bar{V}\!-\!V)^{-1}(V\!+\!\ii\un)\big\}\nn\\
    &&\quad-\frac{k}{4}\tr
      \big \{B[(\bar{V}\!\!+\!\ii\un)(\bar{V}\!\!-\!\!V)^{-1}
      (V\!\!-\!\ii\un)\!-\!(V\!\!-\!\ii\un)(V\!\!+\!\ii\un)^{-1}
      (\bar{V}\!\!+\!\ii\un)(\bar{V}\!\!-\!V)^{-1}(V\!\!+\!\ii\un)]\big\}
  \end{eqnarray*}
  which simplifies to \eqref{ABCDEE}:
  \begin{eqnarray}
      \label{ABCDEE}
&&\quad\quad\quad \mc{H}_V\!\! =\! k\tr A\!+\!
\frac{k}{4\ii}\tr\big\{\tm[(\bar{V}\!-\!\ii\un)(\bar{V}\!\!-\!\!V)^{-1}(V\!-\!\ii\un)\! +\!(\bar{V}\!+\!\ii\un)(\bar{V}\!\!-\!\!V)^{-1}(V\!+\!\ii\un)]\big\}\\
&&\quad\quad\quad\quad\quad\quad\quad\quad+\frac{k}{4}\tr\big\{\tn[(\bar{V}\!-\!\ii\un)(\bar{V}\!\!-\!\!V)^{-1}(V\!-\!\ii\un)\! -\!(\bar{V}\!+\!\ii\un)(\bar{V}\!\!-\!\!V)^{-1}(V\!+\!\ii\un)]\big\} \nn\\
  &&+\frac{k}{4\ii}\tr\big\{A[(\bar{V}\!+\!\ii\un)(\bar{V}\!\!-\!\!V)^{-1}(V\!-\!\ii\un)\! +\!(V\!\!-\!\!\ii\un)(V\!\!+\!\ii\un)^{-1}(\bar{V}\!+\!\ii\un)(\bar{V}\!\!-\!\!V)^{-1}(V\!\!+\!\!\ii\un) ]\big\} \nn\\
   &&-\frac{k}{4}\tr
             \big\{B[(\bar{V}\!+\!\ii\un)(\bar{V}\!-\!V)^{-1}(V\!\!-\!\!\ii\un)\!-\!(V\!\!-\!\!\ii\un)(V\!+\!\ii\un)^{-1}(\bar{V}\!+\!\ii\un)(\bar{V}\!\!-\!\!V)^{-1}(V\!+\!\ii\un)]\big\}.\nn
\end{eqnarray}
\end{Proposition}

\begin{Remark} The simplification above was inspired by the proof  of equation \eqref{VMAARE}
  appearing in \cite [(5.4)]{nou} 
   \begin{equation}\label{VMAARE}
      (\un-W\bar{W})^{-1}=\frac{1}{2\ii}(\bar{V}-\ii\un)(\bar{V}-V)^{-1}(V+\ii\un).
      \end{equation}
   \begin{proof}
    We use the set of identities \eqref{ident}:
  \begin{eqnarray}
  \label{ident}  
 &&\un-(V-\ii\un)(V+\ii\un)^{-1}=2\ii\un(V+\ii\un)^{-1},\nn\\
&& \un+(V-\ii\un)(V+\ii\un)^{-1}=2V(V+\ii\un)^{-1},\\
&& V(V+\ii\un)^{-1}=(V+\ii\un)^{-1}V \nn
\end{eqnarray}
and we find:
\begin{subequations}
  \begin{eqnarray*}
    \un-W\bar{W} &=&\un-(V-\ii \un)(V+\ii
                   \un)^{-1}(\bar{V}+\ii\un)(\bar{V}-\ii\un)^{-1},\\ &=&(\bar{V}-\ii\un)(\bar{V}-\ii\un)^{-1}-(V-\ii
                   \un)(V+\ii
                   \un)^{-1}(\bar{V}+\ii\un)(\bar{V}-\ii\un)^{-1},\\
    &=&\{[\un]\!-\!(V\!-\!\ii\un)(V\!+\!\ii\un)^{-1}\!-\!\ii[\un\!+\!(V\!-\!\ii\un)(V\!+\!\ii\un)^{-1}]\}(\bar{V}\!-\!\ii\un)^{-1}.
   \end{eqnarray*}  
 \end{subequations}

\end{proof}
\end{Remark}

\noindent Simplifying \eqref{ABCDEE} even more, the energy function $\cH_V$ becomes:
\begin{eqnarray}
\label{ABCD1E}
&&\quad\mc{H}_V\! =\!k\tr A\!+\!
\frac{k}{4\ii}\tr\{\tm[(\bar{V}\!-\!\ii\un)(\bar{V}\!\!-\!\!V)^{-1}(V\!-\!\ii\un)\! +\!(\bar{V}\!+\!\ii\un)(\bar{V}\!\!-\!\!V)^{-1}(V\!+\!\ii\un)]\} \\
&&\quad\quad\quad\quad\quad\quad\quad+\frac{k}{4}\tr\{\tn[(\bar{V}\!-\!\ii\un)(\bar{V}\!\!-\!\!V)^{-1}(V\!-\!\ii\un)\! -\!(\bar{V}\!+\!\ii\un)(\bar{V}\!\!-\!\!V)^{-1}(V\!+\!\ii\un)]\} \nn\\
        &&+\frac{k}{4\ii}\tr\{A[2(\bar{V}\!+\!\ii\un)(\bar{V}\!\!-\!\!V)^{-1}V\!-2\ii(V+\ii\un)^{-1}(\bar{V}\!\!+\!\!\ii\un)(\bar{V}-V)]^{-1}(V+\ii\un)\} \nn\\
        &&-\frac{k}{4}\tr\{B[-2\ii(\bar{V}\!+\ii\un)(\bar{V}\!-V)^{-1}\!+2\ii(V\!+\ii\un)^{-1}(\bar{V}\!+\ii\un)(\bar{V}\!-V)^{-1}(V\!+\ii\un)]\}. \nn
  \end{eqnarray}

\begin{Remark}
Substituting $V=x+\ii y$ in \eqref{ABCD1E} and using identities \eqref{ident}, by direct computation we get:
  \begin{subequations}
    \label{ABCD2E}
  \begin{align}
\mc{H}_V& =k\tr A\!+\!\frac{k}{4}\tr\big\{\tm(xy^{-1}x\!+\!y\!-\!y^{-1})\!+\!\tn(xy^{-1}\!+\!y^{-1}x)+ AC -BD\big\},\\
  C &:=xy^{-1}x+y+y^{-1}-2\un,~~~D :=xy^{-1}-y^{-1}x
    \end{align}
  \end{subequations}
and, finally, $\mc{H}_V$ can be written in the following equivalent forms:
\begin{equation}
\label{HV1}
\mc{H}_V=\frac{k}{4}\tr\big\{(\tm+A)(xy^{-1}x\!+\!y\!-\!y^{-1})\!+\!(\tn-B)(xy^{-1}\!+\!y^{-1}x)+2A(y^{-1}+\un)+2B y^{-1}x\big\} \nn
\end{equation}
\begin{equation}
\mc{H}_V\!=\!\frac{k}{4}\!\tr\big\{\tm(xy^{-1}x+y-y^{-1})+\tn(xy^{-1}+y^{-1}x) +A(xy^{-1}x+y+y^{-1}+2\un)-B(xy^{-1}-y^{-1}x)\big\}\nn
\end{equation}
\begin{equation}
\label{HV3}
~\mc{H}_V\!=\!\frac{k}{4}\!\tr\big\{x(\tm+\!A)xy^{-1} \!+ (\tm+\!A)y -(\tm-\!A)y^{-1}\!+(\tn-\!B)x y^{-1}\!+x(\tn-\!B)y^{-1}+2A\big\}~,
\end{equation}
where we kept in mind that $B$ was an antisymmetric matrix, i.e. $B^t=-B$.
\end{Remark}

\begin{Remark}
When substituting $\eta=q+\ii p$, with $q,p\in \cM(n,1,\R)$, and relations \eqref{epsmn1} and \eqref{epsmn2} into \eqref{Hn-eta}, the energy function $\cH_\eta$ becomes:
\begin{equation}
\label{bumber1}
\quad \cH_\eta:=\mc{H}(\alpha,\beta,q,p)
  =\nu\big[2(\alpha^t q-\beta^t p)+q^t(\tm+A)q-p^t(\tm-A)p-2q^t(\tn-B)p\big],
\end{equation}
where $\alpha,\beta, p,q\in \cM(n,1,\R)$ are column vectors.
We show how we obtain relation \eqref{bumber1} from  expression  \eqref{Hn-eta} for $\mc{H}_{\eta}$.
\begin{proof}
 In the expression \eqref{Hn-eta} we have:
   $$\epsilon ^t\eta=(\alpha^t+\ii\beta^t)(q+\ii
  p)=\alpha^t q-\beta^t p+\ii(\alpha^t p+\beta^t q),$$
  and thus:
  $$\epsilon ^t\eta+c.c.=2(\alpha^t q-\beta^t p).$$
  We also have:
  $$  \eta^t\epsilon_-\eta+c.c.=(q^t+\ii p^t)(\tm+\ii \tn)(q+\ii p)+c.c.=2q^t(\tm q-\tn p)-2p^t(\tm p+\tn q).$$
The last term in \eqref{Hn-eta} writes:
    \begin{eqnarray*}  
  \bar{\eta} ^t\epsilon_0\eta&=&2(q^t-\ii  p^t)(A-\ii B)(q+\ii p)\nn\\
      &=&2\big[(q^tA-p^tB)q+(q^tB+p^tA)p+\ii(q^tA\,p-p^tA\,q-q^tB\,q-p^tB\,p)\big]    \nn\\
          & =&2[q^tA q + p^t A p +2 q^tB p]~,\nn
        \end{eqnarray*}
since $p^t A q =q^t A p$, $p^t B q = -q^t B p$ and $q^t B q=p^t B p=0$ due to the fact that $A^t=A$ and $B^t=-B$.

Then the expression for \eqref{Hn-eta} reads:
\begin{eqnarray*}
\mc{H}(\alpha,\beta,q,p)\!\!&=&\!\! \nu\big[2 (\alpha^t q\!-\!\beta^t p)\!+\!q^t(\tm q-\tn p)\!-\!p^t(\tm p+\tn q)\!+\!q^tAq\!+\!p^tAp+2q^tBp\big]\nn\\
\!\!&=&\!\! \nu\big[ 2(\alpha^t q\!-\!\beta^t p)+q^t(\tm+A)q-p^t(\tm-A)p-2q^t(\tn-B)p \big]\nn
\end{eqnarray*}
since $q^t \tn p= p^t \tn q$. Thus we get equation \eqref{bumber1}. 
  \end{proof}
\end{Remark}

\begin{Remark} To check compatibility between energy function $\cH_V$ in \eqref{HV3}, defined for $n>1$, and energy function $\cH(v)$ in \eqref{sup2} defined for $n=1$, we particularize \eqref{HV3} for $n=1$, and, given that when $n=1$ we have $A=c$, $B=0$, $\alpha=a$, $\beta=b$, we find:
\begin{equation}
\label{HVn1}
\mc{H}_V\big|_{n=1}\!=\!\frac{k}{\textcolor{red}{4}y}\big[(\tm+c)(x^2 +y^2) -(\tm-c)\!+2 \tn x\textcolor{red}{+2cy}\big].
\end{equation}
 The differences between \eqref{HVn1} and \eqref{sup2} consist in $k\to\frac{k}{4}$ and in an additional constant term $\frac{1}{2}kc$ in \eqref{HVn1}. 
 
Similarly, testing the compatibility  between energy function $\cH_\eta$ in \eqref{bumber1}, defined for $n>1$, and energy function $\cH(\eta)$ in \eqref{sup1} defined for $n=1$, we particularize  \eqref{bumber1} for $n=1$ and find:
 \begin{equation}
 \label{Heta-n1}
   \mc{H}_\eta\big|_{n=1}= \nu[ 2(a q\textcolor{red}{-}b p)+q^2(\tm+c)-p^2(\tm-c)\textcolor{red}{-}2pq\,\tn].
 \end{equation}
 The differences between \eqref{Heta-n1} and \eqref{sup1} consist in two different signs, i.e. a minus sign in the firrst bracket and a minus sign in front of the last term.
 
All differences mentioned above, marked in red color, are only due to some differences of conventions considered in defining the energy function associated to the Hamiltonian for $n=1$  versus that for $n>1$, i.e. $k\to \frac{k}{4}$, $\eta\to\bar\eta$, $w\to\bar W$which can be noticed at a closer look to \eqref{sup} (taken as in \cite{nou}) and \eqref{enf} (taken as in \cite{FC}). 
\end{Remark}

\vspace{0.5em}

Following \cite{paul} and looking at the definition of an ACOS manifold in the Appendix, we can endow the manifold  $\tilde{\mc{X}}^J_n$ of dimension $n(n\!+\!3)\!+\!1$ with an ACOS structure, $(M_{2N+1},\theta,\omega) $, which in this case can also be considered as a GTACOS structure, since $\omega$ is a symplectic two-form: 
  $$\dd \omega=0.$$
The correspondence between $\tilde{\mc{X}}^J_n$ and $M_{2N+1}$, dimensionaly speaking, gives:
\begin{equation}
\label{dim}
 n(n+3)+1=2N+1 ~\Longrightarrow~ N=\frac{n(n+3)}{2}.
 \end{equation}

 \begin{lemma}\label{L2}
 We endow the extended Siegel-Jacobi upper half-plane $\tilde{\mc{X}}^J_n$ of dimension $n(n+3)+1$,
parametrized by $(x,y,q,p,\kappa)$, with an almost cosymplectic (ACOS) structure $(M_{2N+1},\theta,\omega)$,
where the 1-form $\theta$ and the symplectic 2-form $\omega$ are given by:
 \begin{subequations}   
    \begin{align}
  \theta & =\sqrt{\delta}(\dd \kappa -p^t\dd q+q^t\dd p),\quad \delta>0\label{tn1}\\
  \omega &=\omega_1+\omega_2=-\frac{k}{4}\tr (\dd x\wedge
  \dd y^{-1})+2\nu \dd q ^t\wedge \dd p,\quad y>0\label{tn2}
 \end{align}
  \end{subequations}
where $p,q \in \cM(n,1,\R)$ are column vectors, $x,y \in\cM(n,\R)$ are symmetric matrices and $\kappa\in \R$ is a smooth real function, satisfying the following:
  \begin{equation}
  \label{thOM2}
  \theta\wedge\omega^2=-k \nu\sqrt{\delta} \tr(\dd x\wedge \dd y^{-1})\wedge \dd q^t\wedge\dd p\wedge \dd\kappa \neq 0~.
  \end{equation}
In the terminology of Libermann \cite{paul} recalled in Appendix in the definition of the ACOS manifold $(M_{2N+1},\theta,\omega)$, one can write $\theta$ and $\omega$ for the case $(\tilde{\mc{X}}^J_n,\theta,\omega)$ also in the form:
\begin{subequations}
\begin{eqnarray}
\label{t1}
 && \theta=\sum_{I=1}^{N}(a_I\dd Q^I+b_I \dd P^I)+c \dd \kappa, \quad a_I,b_I,c\in \R,  \quad c\neq 0,\\
 \label{tac2}
  && \omega  = \sum_{I=1}^{N}\dd Q_I\wedge \dd P^I,\quad N=\frac{n(n+3)}{2}.
  \end{eqnarray}
\end{subequations}
 Comparing \eqref{tn1} with \eqref{t1} and  \eqref{tn2} with \eqref{tac2}, we can choose the following correspondences:
\begin{subequations}\label{corr2}
\begin{eqnarray}
&&\quad\quad   c\!=\!\sqrt{\delta}~,~~ a_I\!=\!b_I=0 ~,\quad I=1,\dots,N_1,~~N_1=\frac{n(n+1)}{2}~,\\
&&\quad\quad a_{N_1+j}\!=\! -\frac{\sqrt{\delta}}{2\nu}\,p_j~,~~ b_{N_1+j}\!=\!\sqrt{\delta}\,q_j~,~~ j\!=\!1,~\dots,n~, ~N_1=\frac{n(n+1)}{2}~,\\
&& \quad\quad (Q^I, P^I)  = 
\left\{\begin{array}{ccc} 
~~(\frac{k}{4} x_{ii}, & -(y^{-1})_{ii}), & I=i=1,\ldots,n ;\quad\quad\quad\quad\\
(\frac{k}{2} x_{ij}, & -(y^{-1})_{ij}),& I=n+1,\ldots ,N_1, ~ i,j=1,\ldots,n,~i<j,
  \end{array} \right. \\
&&\quad\quad Q_I \leftrightarrow 2\nu q_i~,~~ P^I\leftrightarrow p^i~,\quad~ I=N_1+i,\ldots,N, ~~i=1,\ldots,n, ~N_1=\frac{n(n+1)}{2}~.
\end{eqnarray}
\end{subequations}
 
  The manifolod $(\tilde{\mc{X}}^J_n,\theta\!=\!\eqref{tn1},\omega\!=\!\eqref{tn2})$ can also be considered a GTACOS (generalized transitive almost cosymplectic) manifold
 since the following condition holds:
  \begin{equation}
  \label{cond02}
 \dd \omega=0 ~.
    \end{equation}
\end{lemma}

\begin{proof} To deduce correspondences \eqref{corr2}, we start with the correspondences between \eqref{tn2} and \eqref{tac2}.  
Since $\omega=\omega_1+\omega_2$, one can write:
\begin{equation}
\label{OM12Darb}
 \omega_1=\sum_{I=1}^{N_1}\dd Q_{I}\wedge \dd P^{I}~,
\quad  \omega_2=\sum_{I=N_1+1}^{N=\frac{n(n+3)}{2}}\dd Q_{I}\wedge \dd P^{I}.
\end{equation}
We compare \eqref{OM12Darb} with \eqref{tn2}, where:
\begin{equation}
\label{OM12}
\omega_1=-\frac{k}{4}\tr (\dd x\wedge
  \dd y^{-1})~,~\quad \omega_2=2\nu \dd q ^t\wedge \dd p=2\nu\sum_{i=1}^n \dd q_i \wedge \dd p^i~,
\end{equation}
with $p,q \in \cM(n,1,\R)$ column vectors, and $x,y \in\cM(n,\R)$ symmetric matices.

Counting the terms in $\omega_2$ we note that:
\begin{equation}
\label{N1}
 N_1=\frac{n(n+3)}{2} -n=\frac{n(n+1)}{2}~,
\end{equation}
and thus for the correspondence between $\omega_2$ in \eqref{OM12Darb} and $\omega_2$  in \eqref{OM12} we can choose:
\begin{equation}
\label{corOM2}
Q_{N_1+i}=2\nu q_i~,\quad P^{N_1+i}=p^i~,\quad~\forall i=1,\ldots,n.
\end{equation}
Let's look now at the expression for $\omega_1$ in \eqref{OM12}, which can be expanded as:
\begin{eqnarray}
  && \omega_1=-\frac{k}{4}\big[ \dd x_{11}\wedge \dd(y^{-1})_{11} +  \dd x_{12}\wedge \dd (y^{-1})_{21} + ...+\dd x_{1n}\wedge \dd(y^{-1})_{n1}    \\
 &&  \quad \quad+ \dd x_{21}\wedge\dd (y^{-1})_{12} + \dd x_{22}\wedge \dd (y^{-1})_{22} +...+
   \dd x_{2n}\wedge \dd(y^{-1})_{n2}+...\nn\\
   &&  \quad \quad+ \dd x_{n1}\wedge \dd (y^{-1})_{1n}+...+ \dd x_{nn}\wedge\dd(y^{-1})_{nn}    \big]=\sum_{i,j=1}^n\dd x_{ij}\wedge\dd(y^{-1})_{ji} \nn\\
  &&  \quad \quad =-\frac{k}{4}\sum_{i=1}^n \dd x_{ii}\wedge\dd(y^{-1})_{ii}-\frac{k}{2}\sum_{i>j=1}^{n} \dd x_{ij}\wedge\dd(y^{-1})_{ji}~.\nn
\end{eqnarray}
We have $n^2$ terms in $\omega_1$, out of which $\frac{n(n+1)}{2}$ are independent, due to the fact that $x$ and $y$ are symmetric matrices. We can split $\omega_1$ into two parts:
\begin{equation}
\label{2parts}
\omega_1=\omega_1'+\omega_1'',~\quad \omega_1'=-\frac{k}{4}\sum_{i=1}^n \dd x_{ii}\wedge\dd(y^{-1})_{ii}~,\quad
\omega_1''=-\frac{k}{2}\sum_{i>j=1}^{n} \dd x_{ij}\wedge\dd(y^{-1})_{ji}~,
\end{equation}
where we see that $\omega_1'$ has $n$ terms, while $\omega_1''$ has $\frac{n(n-1)}{2}$. Summing up we get the $\frac{n(n+1)}{2}$ independent terms of $\omega_1$.

We recall \cite[\S 7]{lutke} that for a matrix $A=(a_{ij})\in \cM(n,\C)$  the vectorisation operator is defined as:
\[
  (\mr{vec}(A)^t)^t=[a_{11},a_{12},\dots,a_{1n},a_{21},\dots,a_{2n},a_{31},\dots,a_{n1},\dots,a_{nn}]\in \cM(1,n^2,\C),
\]
while the half-vectorisation operator is: 
  \[
    (\mr{vech}(A)^t)^t=[a_{11},a_{12},\dots, a_{1n},
    a_{22},\dots,a_{2n},a_{33},\dots, a_{nn}]\in \cM(1,N_1,\C). \\
  \]

Now, since matrices $x,y$ are symmetric, and remembering the decomposition \eqref{2parts} we can write:
\begin{eqnarray}
\label{omega1}
&&\omega_1=-\frac{k}{4}\tr\Big(\dd (\mr{vec}(x)^t)^t\wedge \dd(\mr{vec}(y^{-1})^t)\Big)=-\frac{k}{4}\sum_{r=1}^{n^2}\dd (\mr{vec}(x)^t)_r\wedge \dd(\mr{vec}(y^{-1})^t)^r\\
&&\quad=-\frac{k}{4} \sum_{i=1}^{n}\dd x_{ii}\wedge\dd {(y^{-1}})_{ii}-\frac{k}{2} \sum_{j>i=1}^{n}\dd x_{ij}\wedge\dd (y^{-1})_{ji}.\nn
\end{eqnarray}
We can now write the correspondences between $\omega_1$ in \eqref{omega1} and $\omega_1$ in \eqref{OM12Darb} as follows:
\begin{subequations}
\label{corOM1}
\begin{eqnarray}
&&Q^I\leftrightarrow\frac{k}{4}x_{ii}~,\quad P^I\leftrightarrow -(y^{-1})_{ii}~,
\quad I=i=1,\ldots,n.\\
&&Q^{I}\leftrightarrow\frac{k}{2}x_{ij}~,\quad P^{I}\leftrightarrow-(y^{-1})_{ji}~,\quad  I=n+1,\ldots ,N_1, ~ i,j=1,\ldots,n,~i<j.
\end{eqnarray}
\end{subequations}

Using now correspondences \eqref{corOM2} and \eqref{corOM1} in expression \eqref{tn1} for $\theta$, we find also the correspundences between the coefficients:
$$  c\!=\!\sqrt{\delta},~ a_I\!=\!b_I=0 ,~ I=1,\dots,N_1, ~a_{N_1+j}\!=\! -\frac{\sqrt{\delta}}{2\nu}\,p_j~,~ b_{N_1+j}\!=\!\sqrt{\delta}\,q_j~,~  j\!=\!1,\dots,n~,  $$
with $N_1$ as given in \eqref{N1}. 
\end{proof}

\begin{Theorem}
\label{T1}
There is complete correspondence between all expressions in Lemma 2 and Lemma 1, more precisely Lemma 2 gives a generalization of Lemma 1 for $n>1$, up to the sligthly different conventions that we chose for the two cases in order to quote the results in previously published papers. If one chosed the same conventions, one could formulate Lemma 2 in general for $n\geq 1$. To be more precise, the correspondences between Lemma 2 for $(\tilde{\mc{X}}^J_n,\theta,\omega)$ and Lemma 1 for $(\tilde{\mc{X}}^J_1,\theta,\omega)$ when we take $n=1$ in Lemma 2 are as follows:
\begin{itemize}
\item $\theta$ in \eqref{tn1} perfectly reduces to $\theta$ in \eqref{TT11}, since vectors $q,p\in\cM(n,1,\R)$ reduce to variables $q,p\in \R$;
\item $\omega$ in \eqref{tn2} reduces to $\omega$ in \eqref{TT12}  up to  a coefficient, i.e. $k\to 4k$;
\item condition \eqref{thOM2} reduces to \eqref{thOM}, up to $k\to 4k$, since matrices $x,y\in \cM(n,\R)$ reduce to $x=x_{11},~y=y_{11}$ and vectors $q,p\in\cM(n,1,\R)$ reduce to functions $q,p$;
\item $\theta$ in \eqref{t1} reduces to $\theta$ in \eqref{t11} since $N=\frac{n(n+3)}{2}$ goes to $N=2$ for $n=1$;
\item $\omega$ in \eqref{tac2} reduces to $\omega$ in \eqref{t12} since $N$ reduces to $N=2$;
\item corespondences \eqref{corr2} reduce to \eqref{corr1}, up to $k\to 4k$, since $N\to 2$, $N_1\to 1$, $x,y\in\cM(n,\R)$ reduce to only one matrix element $x=x_{11}$, $y=y_{11}$ and vectors $q,p\in\cM(n,1,\R)$ reduce to only one element $q=q_1$, $p=p_1$;
\item condition \eqref{cond02} reduces identically to \eqref{cond0}.
\end{itemize} 
\end{Theorem}

  \section{Equations of motion on  $\tilde{\mc{X}}^J_1$}

\vspace{0.5em}

Recalling \cite[Theorem 1]{SB22} and particularizing it for $n=1$, i.e. $N=2$, we have:

\begin{Theorem}
\label{T2}
Let  $(M_5,\theta=\eqref{t11},\omega=\eqref{t12})$ be an
almost cosymplectic manifold,   and let $\cH\in C^{\infty}(M)$ be a smooth energy function associated to the Hamiltonian.
The equations of motion on this ACOS manifold are:
\begin{equation}\label{ECABC}
  \dot{Q}_I\!=\!\frac{\pa \cH}{\pa P^I}-b_IR(\cH) ,~
\dot{P}_I\!=\!-\frac{\pa \cH}{\pa Q_I}\!+a_IR(\cH),~
\dot{\kappa}\!=\!\frac{1}{c}(\!-\!a_I\frac{\pa \cH}{\pa P^I}+b_I\frac{\pa \cH}{\pa Q^I})\!-\!\cH,~ I=1,2,
\end{equation}
where $R$ is the Reeb vector field and is defined as:
$$R=\frac{1}{c}\frac{\pa}{\pa \kappa}~.$$
\end{Theorem}

\begin{Proposition}
\label{P2}
 Considering the GTACOS (generalized transitive almost cosymplectic) manifold
$(\tilde{\mc{X}}^J_1,\theta,\omega)$ on which we use notations $A_I,~ B_I, ~C$, with $I=1,2$, where
  \begin{equation}
    \dot{Q_I}=A_I, ~\dot{P_I}=B_I,~\dot{\kappa}=C, \label{ABCQ}
    \end{equation}
we have the expressions \cite[Proposition 4]{SB22}:
\begin{subequations}\label{ABCEU}
\begin{align}
A_1&=y^2\frac{\pa  \mc{H}}{\pa y}, \quad   \quad A_2=\frac{\pa \mc{H}}{\pa p}-q \frac{\pa \mc{H}}{\pa \kappa},\\
B_1 & =-\frac{1}{k}\frac{\pa \mc{H}}{\pa x}, ~~\quad B_2=-\frac{1}{2\nu}(\frac{\pa \mc{H}}{\pa q}+p\frac{\pa \mc{H}}{\pa\kappa}) ,\\
C & =\frac{1}{2\nu}(p\frac{\pa \mc{H}}{\pa p}+q\frac{\pa\mc{H}}{\pa q})-\mc{H},~~ 
\end{align}
\end{subequations}
where $\mc{H}\in\C^{\infty}(\tilde{\mc{X}}^J_1)$ is energy function \eqref{H1-tilde} associated to the linear Hamiltonian \eqref{guru}.

 With formulas \eqref{ABCEU}, the equations of motion  \eqref{ECABC} on the 5-dimensional extended
Siegel-Jacobi half-plane $\tilde{\mc{X}}^J_1$ organized as a GTACOS
manifold $(\tilde{\mc{X}}^J_1,\theta\!=\!\eqref{TT11}, \omega\!=\!\eqref{TT12})$ reduces to \cite[Proposition 4,(4.8), (4.9)]{SB22}:
\begin{subequations}\label{ECXXY}
\begin{align}
\dot{x} &=\frac{y^2}{k}\frac{\pa \mc{H}}{\pa y}, \quad\qquad\qquad \dot{y} = -\frac{y^2}{k}\frac{\pa \mc{H}}{\pa x},\\
\dot{q} &= \frac{1}{2\nu}(\frac{\pa \mc{H}}{\pa p}
   -q\frac{\pa \mc{H}}{\pa \kappa}),\quad
\dot{p}  =-\frac{1}{2\nu}(\frac{\pa \mc{H}}{\pa q}+
p\frac{\pa\mc{H}}{\pa\kappa}),\\
\dot{\kappa}&=\frac{1}{2\nu}(p\frac{\pa \mc{H}}{\pa p} +q\frac{\pa\mc{H}}{\pa q})-\mc{H},~
\end{align}
\end{subequations}
where in \eqref{ECXXY} we take
$\mc{H}=\mc{H}_{\tilde{\mc{X}}^J_1(\eta,v,\kappa)}$ given in \eqref{H1-tilde}. Recalling the expressions for the components $\mc{H}(q,p)$ and $\mc{H}(x,y)$ of $\mc{H}$, given by \eqref{sup1} and \eqref{sup2} respectively, we get the following
equations of motion on $\tilde{\mc{X}}^J_1$:
\begin{subequations}\label{ECXXX}
 \begin{align}
 \dot{x} &=(\tm+c)(y^2-x^2)+2\tn x+\tm-c,\\
 \dot{y} &= -2y[(\tm+c)x-\tn], \\
\dot{q} &= (c-\tm)p+\tn q+b-\frac{q}{2\nu}\frac{\pa \mc{H}}{\pa \kappa}, \\
\dot{p} &=-[(\tm+c)q+\tn p+a]-\frac{p}{2\nu}\frac{\pa\mc{H}}{\pa \kappa},\\
   \dot{\kappa}&=(c-\tm)p^2+(c+\tm)q^2+2\tn pq+pb+aq- \mc{H}.~~
\end{align}
\end{subequations}
\end{Proposition}

 \begin{proof} 
 
Calculating partial derivatives  in  \eqref{ECXXY} for energy function \eqref{H1-tilde} proves \eqref{ECXXX}:
    \begin{eqnarray*}
      \frac{1}{k}\frac{\pa \mc{H}}{\pa x}& =&\frac{2}{y}[(\tm+c)x-\tn],\\
     \frac{1}{k}\frac{\pa \mc{H}}{\pa y} &=&\frac{1}{y^2}[(\tm+c)(y^2-x^2)+2\tn x+m-c],\\
  \frac{1}{2\nu}\frac{\pa \mc{H}}{\pa q}& =&(\tm+c)q+\tn p+a, \\
   \frac{1}{2\nu}\frac{\pa \mc{H}}{\pa p} & =&(c-\tm)p+\tn q+b.\\
   \end{eqnarray*}
  \end{proof}

\vspace{-1.2em}

\section{Equations of motion on $\tilde{\mc{X}}^J_n$, when $n>1$}

We recall again \cite[Theorem 1]{SB22} which here becomees:

\begin{Theorem}
\label{T3} 
Let  $(M_{2N+1},\theta=\eqref{t1},\omega=\eqref{tac2})$ be an
almost cosymplectic (ACOS) manifold,   and let $\cH\in C^{\infty}(M)$ be a smooth energy function associated to the Hamiltonian.
The equations of motion on this ACOS manifold are:
\begin{equation}\label{ECABCC}
  \dot{Q}_I\!=\!\frac{\pa \cH}{\pa Q^I}-b_IR(\cH) ,~
\dot{P}_I\!=\!-\frac{\pa \cH}{\pa Q_I}\!+a_IR(\cH),~
\dot{\kappa}\!=\!\frac{1}{c}(\!-a_I\frac{\pa \cH}{\pa P^I}\!+b_I\frac{\pa \cH}{\pa Q^I})\!-\!\cH,~I=\overline{1,N},
\end{equation}
where $R=\frac{1}{c}\frac{\pa}{\pa \kappa}$ is the Reeb vector.
\end{Theorem}

\begin{Proposition}
\label{P3}
 If the  energy function \eqref{enf} on the extended Siegel-Jacobi upper half space $\tilde{\mc{X}}^J_n$
 associated to  linear   Hamiltonian \eqref{HACA} is
 $\mc{H}=\mc{H}(\alpha,\beta,q,p)+\mc{H}_V(x,y)+\mc{H}(\kappa)$,
 where $\mc{H}(\alpha,\beta,q,p)$  and $\mc{H}_V(x,y)$  are given by \eqref{Hn-eta}
 and by \eqref{Hn-W} respectively,
 then the  equations of motion  on the extended
 Siegel-Jacobi space $(\tilde{\mc{X}}^J_n, \theta=\eqref{tn1}, \omega=\eqref{tn2})$
 generated by the linear Hamiltonian $\mc{H} $ are given by the  relations \eqref{ECABCC} with correspondences \eqref{corr2} and become:
 \begin{subequations}
 \label{eom}
\begin{eqnarray}
&& \dot{x}_{ij} =- \frac{2}{k} \frac{\pa \mc{H}}{\pa y^{-1}_{ij}}= \frac{2}{k}(y^2)_{ik} \frac{\pa \mc{H}}{\pa y_{kj}},\quad  \dot{x}_{ii} =\frac{4}{k}(y^2)_{ik} \frac{\pa \mc{H}}{\pa y_{ki}},\quad 0<i<j=2\ldots,n,\\
 &&   \dot{y}_{ij} = -\frac{2}{k}(y^2 )_{ik}\frac{\pa \mc{H}}{\pa x_{kj}}, \quad \dot{y}_{ii}= -\frac{4}{k}(y^2 )_{ik}\frac{\pa \mc{H}}{\pa x_{ki}},\quad 0<i<j=2\ldots,n,\\
&&\dot{q}_i=\frac{1}{2\nu}\Big(\frac{\pa \cH}{\pa p^i} -q_i \frac{\pa \cH}{\pa \kappa} \Big),~
\quad \dot{p}_i=- \frac{1}{2\nu}\Big(\frac{\pa \cH}{\pa q^i} +p_i \frac{\pa \cH}{\pa \kappa} \Big),\quad i=1,\ldots,n,\\
&&\dot\kappa=\frac{1}{2\nu}\Big(p_i\frac{\pa \cH}{\pa p^i} +q_i \frac{\pa \cH}{\pa q^i} \Big)-\cH~,\quad i=1,\ldots,n.
\end{eqnarray}
\end{subequations}
where $x_{ij}, y_{ij}$, with $i,j=1,\dots, n$, are elements of real symmetric matrices,
while $q_i, p_i$, with $i=1,\ldots,n$, are elements of real $n$-vectors.
 \end{Proposition}
 
\begin{Remark}
In order to calculate $\frac{\pa{\mc{H}}}{\pa{x_{ij}}}$ we will repeatedly use the formula {\rm{(4.5)}} in \cite{nou} for a symmetric matrix of elements $z_{ij}$:
\begin{equation}
\label{mircea}
  \frac{\pa z_{ab}}{\pa z_{ij}}=\delta_{ai}\delta_{bj} 
  + \delta_{aj} \delta_{bi}  -\delta_{ab}\delta_{ij}\delta_{bi}.
   \end{equation}
 When computing  $\frac{\pa{\mc{H}}}{\pa{y_{ij}}}$ we will use the formula:
 \begin{equation}
  \label{derivinv}
 \frac{\pa y^{-1}_{ab}}{\pa y_{ij}}=-y_{a c}^{-2}\frac{\pa y_{c b}}{\pa y_{ij}} .
   \end{equation}
\end{Remark}

\begin{Remark}
Using expression \eqref{HV3} for $\mc{H}_V$, we calculate 
$\frac{\pa{\mc{H}}}{\pa{x_{ij}}}=\frac{\pa{\mc{H}_V}}{\pa{x_{ij}}}$  

\begin{eqnarray}
\frac{\pa{\mc{H}_V}}{\pa{x_{ij}}}&=&\frac{k}{4}\Big[\frac{\pa x_{ab}}{\pa x_{ij}}(\tm+A)_{bc} \,x_{cd} \,y^{-1}_{da}
+x_{ab}(\tm+A)_{bc}\,\frac{\pa x_{cd}}{\pa x_{ij}} \,y^{-1}_{da} \nn\\
&&\quad +(\tn-B)_{ab}\,\frac{\pa x_{bc}}{\pa x_{ij}} y^{-1}_{ca}
+\frac{\pa x_{ab}}{\pa x_{ij}}(\tn-B)_{bc}\, y^{-1}_{ca}\Big].
\end{eqnarray}
With notation $F_{(ij)}=F_{ij}+F_{ij}$, the previous expression reeads:
\begin{eqnarray}
\frac{\pa{\mc{H}_V}}{\pa{x_{ij}}}&=&\frac{k}{4}\Big\{\big[(\tm+A)xy^{-1}\big]_{(ij)}-\big[(\tm+A)xy^{-1}\big]_{ii}\delta_{ij} \nn\\
&&\quad+\big[y^{-1}x(\tm+A)\big]_{(ij)}-\big[y^{-1}x(\tm+A)\big]_{ii}\delta_{ij}\nn\\
  &&\quad +[y^{-1}(\tn-B)]_{(ij)}-[y^{-1}(\tn-B)]_{ii}\delta_{ij}
  + [(\tn-B)y^{-1}]_{(ij)}-[(\tn-B)y^{-1}]_{ii}\delta_{ij}\Big\},
\end{eqnarray}
or equivallently:

\begin{eqnarray}
\frac{\pa{\mc{H}_V}}{\pa{x_{ij}}}&=&\frac{k}{4}\Big[(\tm+A)xy^{-1}+y^{-1}x(\tm+A)+y^{-1}(\tn-B) + (\tn-B)y^{-1}\Big]_{(ij)}\nn\\
&&-\frac{k}{4}\Big[(\tm+A)xy^{-1} +y^{-1}x(\tm+A) +y^{-1}(\tn-B)
+(\tn-B)y^{-1}\Big]_{ii}\delta_{ij}.\nn
\end{eqnarray}

\end{Remark}

\begin{Remark}
Using expression \eqref{HV3} for $\mc{H}_V$, we calculate 
$\frac{\pa{\mc{H}}}{\pa{y_{ij}}}=\frac{\pa{\mc{H}_V}}{\pa{y_{ij}}}$  

  \begin{eqnarray*}
  \frac{\pa \mc{H}_V}{\pa y_{ij}}\!\!\!&=& \!\!\!
  \frac{k}{4}\Big[x(\tm\!+\!A)x \!-\!(\tm\!-\!A)\!+\!(\tn\!-\!B)x \!+\!x(\tn\!-\!B)\Big]_{ab} \frac{\pa y^{-1}_{ba}}{\pa y_{ij}} \!+\! \frac{k}{4} (\tm\!+\!A)_{ab}\frac{\pa y_{ba}}{\pa y_{ij}}\nn\\
  &=& \!\!\!
  \frac{k}{4}\Big\{\Big[x(\tm\!+\!A)x \!-\!(\tm\!-\!A)\!+\!(\tn\!-\!B)x \!+\!x(\tn\!-\!B)\Big]_{ab} y^{-2}_{bc} +(\tm\!+\!A)_{ac}\Big\}\frac{\pa y_{ca}}{\pa y_{ij}}\nn\\
  \end{eqnarray*}
 and we get:
 \begin{eqnarray}
  \frac{\pa \mc{H}_V}{\pa y_{ij}}\!\!\!&=& \!\!\!
   \frac{k}{4}\Big[x(\tm\!+\!A)x y^{-2}\!-\!(\tm\!-\!A)y^{-2}\!+\!(\tn\!-\!B)xy^{-2} \!+\!x(\tn\!-\!B)y^{-2}+(\tm\!+\!A)\Big]_{(ij)}\nn\\
 &-&\frac{k}{4}\Big[x(\tm\!+\!A)x y^{-2}\!-\!(\tm\!-\!A)y^{-2}\!+\!(\tn\!-\!B)xy^{-2} \!+\!x(\tn\!-\!B)y^{-2}+(\tm\!+\!A)\Big]_{ii}\delta_{ij},\nn\\
  \end{eqnarray}
   where again we used the simplification $F_{(ij)}=F_{ij}+F_{ji}$.
  \end{Remark}       

\begin{Remark} Recalling expression \eqref{bumber1} for $\mc{H}(\alpha,\beta,q,p)$, and applying the formulas for matrics derivations such as:
$$ \frac{\pa}{\pa q} (q^t A p)=A p ,~
\frac{\pa}{\pa q} (q^t A p)=\frac{\pa}{\pa q} (p^t A^t q)=A^t q~,$$
we compute derivatives with respect to $p$ and $q$:
$$ \mc{H}(\alpha,\beta,q,p):
  =\nu\big[2(\alpha^t q-\beta^t p)+q^t(\tm+A)q-p^t(\tm-A)p-2q^t(\tn-B)p\big], $$
\begin{align*}
\frac{\pa\mc{H}(\alpha,\beta,q,p)}{\pa q}&=\nu\,\big[2\alpha_i+2(m+A)q-2(\tn-B)p \big],\\
\frac{\pa\mc{H}(\alpha,\beta,q,p)}{\pa p}&=-\nu\,\big[2\beta+2(\tm-A)p_i+2(\tn+B)q\big],
\end{align*}
which on components become:
\begin{align*}
\frac{\pa\mc{H}(\alpha,\beta,q,p)}{\pa q_i}&=2\nu\,\big[\alpha_i+(m+A)q_i-(\tn-B)p_i \big],\\
\frac{\pa\mc{H}(\alpha,\beta,q,p)}{\pa p_i}&=-2\nu\,\big[\beta_i+(\tm-A)p_i+(\tn+B)q_i\big].
\end{align*}
\end{Remark}

 \begin{Theorem}
 \label{T4}
Endowing the extended Siegel-Jacobi  space $(\tilde{{\mc{X}}}^J_n, \theta, \omega)$ of dimension $\dim \tilde{{\mc{X}}}^J_n=n(n+3)+1$ with an  ACOS  structure given by $\theta=\eqref{t1}$ and $\omega=\eqref{tac2}$, we find that  Hamiltonian \eqref{HACA} to which we associate the energy function \eqref{enf} generates the following equations of motion on $\tilde{\mc{X}}^J_n$ for variables $(x,y,q,p,\kappa)$, where $x,y\in\cM(n,\R)$, $q,p\in\cM(n,1,\R)$, $\kappa\in\R$:
\begin{eqnarray*}
\dot x_{ij}\!\!\!&=&\Big[ x(\tm\!+\!A)x \!-\!(\tm\!-\!A)\!+\!(\tn\!-\!B)x \!+\!x(\tn\!-\!B)+(\tm\!+\!A)y^2\Big]_{(ij)}\nn\\
 &-&\Big[ x(\tm\!+\!A)x \!-\!(\tm\!-\!A)\!+\!(\tn\!-\!B)x \!+\!x(\tn\!-\!B)+(\tm\!+\!A)y^2\Big]_{jj}\delta_{ij},\\
\dot{y}_{ij}\!\!&=&\Big[(\tm+A)xy +y^{-1}x(\tm+A)y^2 +y^{-1}(\tn-B)y^2
+(\tn-B)y\Big]_{ij}\nn\\
&-&\Big[(\tm+A)xy +x(\tm+A)y +(\tn-B)y
+(\tn-B)y\Big]_{jj}\delta_{ij},\nn\\
\!\!\dot{q_i} & =&- \beta_i-(\tm-A)p_i-(\tn+B)q_i -q_i\frac{1}{2\nu}\frac{\pa \mc{H}}{\pa \kappa},\\
\!\!\dot{p_i} &= &-\alpha_i-(\tm+A)q_i+(\tn-B)p_i-  p_i\frac{1}{2\nu}\frac{\pa \mc{H}}{\pa \kappa},\\
  \!\!\dot\kappa &=&- p^i\big[\beta_i+(\tm-A)p_i+(\tn+B)q_i \big]
 + q^i\big[\alpha_i+(\tm\!+\!A)q_i\!-\!(\tn\!-\!B)p_i\big]  \!-\!\mc{H}.
\end{eqnarray*}
 
 \

\end{Theorem}

\begin{Remark}The equations of motion given in Theorem 5 for the case $(\tilde{{\mc{X}}}^J_n, \theta, \omega)$ reduce, when particularizing them to $n=1$, to equations of motion for $(\tilde{{\mc{X}}}^J_1, \theta, \omega)$ written in Proposition 2,  in \eqref{ECXXX}, but only up to some small differences generated by using different conventions when writting the energy functions for $n=1$ and $n>1$, more concretely due to the correspondences: $\eta\leftrightarrow\bar\eta$, $w\leftrightarrow W$ and $k\to\frac{k}{4}$. Thus, taking $n=1$ in Theorem 5 we get the following expressions, where we marked in red color the differences from \eqref{ECXXX} due to other conventions taken:
\begin{subequations}\label{eom1}
 \begin{align}
 \dot{x} &=(\tm+c)(x^2\textcolor{red}{+}y^2)+2\tn x\textcolor{red}{-}(\tm-c),\\
 \dot{y} &= \textcolor{red}{+}2(\tm+c)xy+2\tn y, \\
\dot{q} &= \textcolor{red}{+}b-(\tm-c)p\textcolor{red}{+}\tn q-\frac{q}{2\nu}\frac{\pa \mc{H}}{\pa \kappa}, \\
\dot{p} &=-a-(\tm+c)q\textcolor{red}{-}\tn p-\frac{p}{2\nu}\frac{\pa\mc{H}}{\pa \kappa},\\
   \dot{\kappa}&=-(\tm-c)p^2+(\tm+c)q^2\textcolor{red}{+}2\tn pq+aq\textcolor{red}{+}pb- \mc{H}.~~
\end{align}
\end{subequations}
\end{Remark}

\vspace{0.5em}

\section{Conclusions}

In this paper we make a parallel between the Siegel-Jacobi upper half-space $\tilde{\mc{X}}^J_n$, with $n>1$, and the Siegel-Jacobi upper half-plane $\tilde{\mc{X}}^J_1$ regarding their Hamiltonians linear in the generators of the Jacobi group, regarding the energy functions associated and the corresponding equations of motion in each of the two cases.

Starting from notions and formulas given in previous publications \cite{FC}--\cite{SB22}, \cite{SB12U}--\cite{SB25}, the present paper also gives new and very important results. Its content is structured into 6 chapters and an Appendix. In Chapter 2 we recall Lemma 1 for the extended Siegel-Jacobi upper half-plane $\tilde{\mc{X}}^J_1$, in Chapter 3 we formulate Lemma 2 for the extended Siegel-Jacobi space $\tilde{\mc{X}}^J_n$ and Theorem 1 that shows correspondences between Lemma 1 and Lemma 2. In Chapter 4 we write Theorem 2 and Proposition 2 in which we give the formulas and respectively the concrete expressions for the equations of motion on $\tilde{\mc{X}}^J_1$ organized as an ACOS manifold, while in Chapter 5 we give Theorem 3, which is a generalization on $\tilde{\mc{X}}^J_n$ of Theorem 2. Other new results regarding the equations of motion on $\tilde{\mc{X}}^J_n$ are then given in Proposition 3 and Theorem 4.

\vspace{0.5em}

Thus, the most important contributions of this paper consist in:
\begin{itemize}
 \item the expressions for the energy function $\cH_V$, given in \eqref{HV3}, and the energy function $\cH_\eta=\cH(\alpha,\beta,q,p)$, given in \eqref{bumber1}, both for $n>1$;
 \item Lemma \ref{L2} formulated for $\tilde{\mc{X}}^J_n$, when $n>1$, which gives a complete description of the manifold $(\tilde{\mc{X}}^J_n,\theta =\eqref{tn1},\omega = \eqref{tn2})$ endowed with an ACOS/GTACOS structure, and the correspondences \eqref{corr2} mentioned therein;
 \item Theorem \ref{T1} that shows the perfect correspondence between Lemma \ref{L2} for the case $n>1$ and Lemma \ref{L1} for the case $n=1$;
 \item Theorem \ref{T3}, which gives the general formulas for the equations of motions on $(\tilde{\mc{X}}^J_n,\theta =\eqref{tn1},\omega = \eqref{tn2})$, for $n>1$, viewed as an ACOS manifold. This theorem is a generalization of Theorem \ref{T2} formulated for $n=1$;
 \item Proposition \ref{P3}, which expands the formulas given in Theorem \ref{T3} for the concrete correspondences found in Lemma \ref{L2};
 \item Theorem \ref{T4} which gives the concrete form of the equations of motion on the extended Siegel-Jacobi space $\tilde{\mc{X}}^J_n$ for the energy function \eqref{enf} associated to the linear Hamiltonian \eqref{HACA}.
\end{itemize}

\

\section*{Appendix}

\

{\bf The energy function} $\mc{H}$ (or the classical Hamiltonian)  attached to
  the quantum Hamiltonian $\mb{H}$ is, as in \cite{berezin}:
    $$\mc{H}(z,\bar{z})=<e_{\bar{z}},e_{\bar{z}}>^{-1}<e_{\bar{z}}|\mb{H}|e_{\bar{z}}>.$$
We consider an algebraic Hamiltonian linear in the generators
${{X}}_{\lambda} $ of the group of symmetry $G$:
\begin{equation}\label{lllu}
H=\sum_{\lambda\in\Delta}\epsilon_{\lambda}{{X}}_{\lambda} ,
\end{equation}
and the associated operator $\mb{H}$:
\begin{equation}\label{llluT}
\mb{H}=\sum_{\lambda\in\Delta}\epsilon_{\lambda}{\mb{X}}^{\dagger}_{\lambda} .
\end{equation}
The energy function associated to the Hamiltonian \eqref{lllu} is:
$$\mc{H}(z,\bar{z})=\sum_{\lambda\in\Delta}\epsilon_{\lambda}\frac{(e_{\bar{z}},\mb{X}_{\lambda}e_{\bar{z}})}{(e_{\bar{z}},e_{\bar{z}})}.$$

We recall below a diagram introduced in \cite{SB22}.

\

\begin{center}
Table:  Geometric structures on odd dimensional manifolds
\[ 
\boxed{
  \begin{array}{cc cc cc cc cc}
     \text{COS}&\subset&\text{GTACOS}&\supset&\text{TACS}&&&\\
   &&\cap&&&&&\\
 & &\text{ACOS}& &&&&\\
     &&\cup&&&&&\\
  \text{CH}&\subset&\text{C}&\subset&\text{ACM}&\subset&\text{SAC}&\supset \text{N=SAS}\\
     &&&&\cap&&&\\
   &&&&\text{ACOK}&&&
 \end{array}
 }
\]
\end{center}

\

{\bf ACOS} 

In \cite{paul}:   {\it  almost cosymplectic manifold} -- is a triplet $(M,\theta,\omega) $,
where $M$ is a $(2N+1)$-dimensional manifold, $\theta\in \got{D}^1$ is a 1-form,
$\omega$ is a 2-form with $\text{rank}(\omega)=2N$, and, in Darboux coordinates $(q_i,p_i,\kappa)_{i=1,\ldots,N}$, they can be written as:
\begin{eqnarray}
\label{t11A}
 && \theta=a_i\dd q^i+b_i \dd p^i+c \dd \kappa, \quad i=1,\ldots,N,\quad a_i,b_i,c\in \R,  \quad c\neq 0,\\
 \label{t12A}
  &&\omega  = \dd q_i\wedge \dd p^i, \quad i=1,\ldots,N,
  \end{eqnarray}
satisfying $~{\rm rank}(\omega)=2N$ and:
\begin{equation}\label{thtOM}
  \theta\wedge\omega^N\neq 0.
\end{equation} 
The {\it Reeb vector} $R\in \got{D}^1$ is defined by the equations:
\begin{equation}\label{reeb}
R\lrcorner\,\omega=0, \quad  
  R\lrcorner\,\theta=1. 
\end{equation}

 In \cite{alb}: the almost cosymplectic manifold
$(M,\theta,\omega) $ of  \cite{paul}  is called {\it almost contact manifold}. \\
It is proved in \cite[Proposition  1]{alb} that {\it  the application} $\flat: TM\rightarrow  TM^*$
\begin{equation}\label{bemol}
X\rightarrow X^{\flat}=X\lrcorner \,\Omega + (X\lrcorner\,\theta)\theta
\end{equation}
{\it is a vector bundle isomorphism}.

{\bf COS}     

In \cite{paul,can}:   {\it  cosymplectic manifold} -- is the triplet  $(M,\theta,\Omega) $,
defined in  {\bf ACOS},  where

\[
  \dd \theta =0, \quad \dd \omega=0.
\]

{\bf GTACOS} 

A {\it generalized transitive almost cosymplectic
manifold}  is an {\bf ACOS} $(M,\theta,\omega)$ with $\dd\omega= 0$,  more general then TACOS and C-structures.

\

{\bf TACS} 

In \cite{alb,can}: {\it transitive almost contact structure} (TACS) - is  an almost contact manifold $(M,\theta,\omega)$, with $\dim M=2N+1$ defined as an {\bf ACOS} with \[\dd
  \omega=0,\] and where around every point of $M$
there is a   neighborhood  where there are local Darboux
coordinates $(\kappa, q^1,\dots,q^n,p_1,\dots, p_n)$ such that
\begin{equation}\label{tac1}
 \theta   =\dd \kappa +\epsilon p_i\dd q^i,\quad
           i=1,\dots,N,~~\epsilon\in \R.
 \end{equation}
To a function $f\in C^{\infty}(M)$ it is associated a Hamiltonian vector field $X_f$
defined by \cite[(3)]{alb} as solution to the
equations:
\begin{subequations}
\begin{align*}
X_f\lrcorner\theta & =\epsilon f,\\
X_f\lrcorner \Omega & = \dd \!f-(R f)\theta.
\end{align*}
\end{subequations}

{\bf CH}  

In \cite{MLEO,MLEO1}: {\it contact Hamiltonian system}  $(M,\eta)$ --
is defined
as the almost cosymplectic manifold $(M,\eta, \dd \eta)$
for which hold both \eqref{thtOM} and:
\begin{equation}
\label{thtOM1}
 \eta\wedge\dd\!\eta^n\neq 0.
\end{equation}

\

{\bf C}

Following \cite{boy},    in \cite{SB19} we used  the definition   of a
{\it 
contact structure}  ($M_{2n+1},\eta)$ when  $\eta\in \mc{D}_1$ satisfies
\eqref{thtOM1}.

The  {\it 
contact structure} can be given by {\it a codimension one subbundle} $\mc{D}$
{\it of the tangent bundle} $TM$ {\it which is as far from being integrable as
possible}, and $\mc{D}:=\text{Ker}(\eta)$.\vspace{3ex}

\textbf{Acknowledgments.} 

This work was supported by the national grant PN 23 21 01 01 / 2023.

\end{document}